\documentclass[a4paper]{amsart}
\usepackage[margin={100pt, 100pt}, centering]{geometry}
\usepackage[utf8]{inputenc}
\usepackage[T1]{fontenc}
\usepackage{textcomp}
\usepackage{amsmath, amssymb}
\usepackage{amsthm}
\usepackage{mathrsfs}
\usepackage{esint}
\usepackage{float}
\usepackage{import}
\usepackage{xifthen}
\usepackage{pdfpages}
\usepackage{transparent}
\usepackage{hyperref}
\usepackage{subcaption}
\usepackage{caption}
\usepackage[most]{tcolorbox}
\usepackage[shortlabels]{enumitem}

\newtheorem{theorem}{Theorem}

\newtheorem{restatetheorem}{Theorem}

\newtheorem{problem}{Problem}[section]
\newtheorem{proposition}[problem]{Proposition}
\newtheorem{corollary}[problem]{Corollary}

\newtheorem{remark}[problem]{Remark}

\theoremstyle{definition}
\newtheorem{definition}[problem]{Definition}
\newtcbox{\redbox}{colback=red!5!white,colframe=red!75!black, tcbox raise base}

\newcommand{\grad}{\nabla}
\newcommand{\R}{\mathbb R}
\newcommand{\T}{\tilde}
\newcommand{\bd}{\partial}
\newcommand{\lap}{\Delta}
\renewcommand{\div}{\operatorname{div}}
\newcommand{\RP}{\mathbb{RP}}

\begin{document}
\title{The Yamabe invariant of $\mathbb{RP}^3$ via Harmonic Functions}
\author{Liam Mazurowski}
\address{Department of Mathematics, Cornell University, Ithaca, NY, 14853}
\email{lmm334@cornell.edu}
\author{Xuan Yao}
\address{Department of Mathematics, Cornell University, Ithaca, NY, 14853}
\email{xy346@cornell.edu}

\maketitle
\begin{abstract}
We use harmonic functions to give a new proof of a result of Bray and Neves on the Yamabe invariant of $\RP^3$. 
\end{abstract}

\section{Introduction}

The Yamabe problem asks whether every closed, Riemannian manifold $(M^n,g)$ admits a conformal metric with constant scalar curvature. This is equivalent to finding a solution to $L\varphi = \lambda \varphi^{(n+2)/(n-2)}$ where 
\[
L = \lap -\frac{n-2}{4(n-1)}R_g
\]
is the so-called conformal Laplacian. Yamabe \cite{Yamabe} claimed to have an affirmative answer to the problem, but his proof was later found to have a gap. 
Trudinger \cite{Trudinger} was able to fix the gap in the case where the conformal Yamabe invariant 
\[
Y(M,[g]) = \inf\left\{-\frac{4(n-1)}{n-2}\frac{\int_M \varphi L \varphi \, dV}{\left(\int_M \varphi^6\, dV\right)^{1/3}}:\, \varphi\in H^1(M)\right\} 
\]
satisfies $Y(M,[g]) \le 0$. Later, Aubin \cite{Aubin} showed that in fact  $Y(M,[g]) < Y(S^n,[g_{0}])$ implies existence of a constant scalar curvature metric conformal to $g$. Here $g_0$ is the round metric on $S^n$. Finally, Schoen \cite{Schoen-84} completely solved  the Yamabe problem by showing that $Y(M,[g]) < Y(S^n,[g_0])$, unless $(M,g)$ is conformal to the round sphere. 
    
We now focus attention on 3-dimensional manifolds. Schoen \cite{Schoen-89} defined the $\sigma$-invariant (or smooth Yamabe invariant) of a closed  manifold $M^3$ by 
\[
\sigma(M) = \sup_g Y(M,[g]). 
\]
If follows from the solution to the Yamabe problem that 
\[
 \sigma(S^3) = Y(S^3,g_{0}) = 6(2\pi^2)^{2/3} =: \sigma_1.
\]
It is also known that $\sigma(S^2\times S^1) = \sigma_1$ and that $\sigma(T^3) = 0$. Bray and Neves \cite{bray2004classification} computed the $\sigma$-invariant of $\RP^3$ and showed that 
\[
\sigma(\RP^3) = \frac{\sigma_1}{2^{2/3}} =: \sigma_2. 
\]
They also computed $\sigma(\RP^2\times S^1) = \sigma_2$ and, more generally, showed that any 3-manifold $M$ satisfying certain topological properties has $\sigma(M) \le \sigma_2$. Akutagawa and Neves \cite{Akutagawa-Neves}
computed that $\RP^3 \# (S^2 \times S^1) = \sigma_2$. To the authors' knowledge, no 3-manifold is known to have positive $\sigma$-invariant different from $\sigma_1$ and $\sigma_2$, although Schoen conjectures that 
\[
\sigma(M) = \frac{\sigma_1}{n^{2/3}}
\]
when $M=S^3/G$ and $\vert G\vert=n$.

To compute $\sigma(\RP^3)$, Bray and Neves argue as follows. Given an arbitrary metric $g$ on $\RP^3$, they first construct an associated asymptotically flat metric $g_{AF}$ with vanishing scalar curvature on $\RP^3$ minus a point. In the model case, where $g$ is the round metric, this construction yields the Schwarzschild metric. Then they use inverse mean curvature flow to transport the optimal test function from the  Schwarzschild manifold to the manifold $(\RP^3-\{p\},g_{AF})$. Finally, they use the properties of inverse mean curvature flow, and in particular, the monotonicity of the Hawking mass, to show that the transported test function witnesses $Y(\RP^3,[g]) \le \sigma_2$. 

Recently, harmonic and $p$-harmonic functions have been found to be useful tools for studying manifolds with non-negative scalar curvature. Stern \cite{Stern} first uncovered the connection between harmonic functions and scalar curvature, and used it to give a new proof that $T^3$ admits no metric of positive scalar curvature, among other things.  Since then, many results on non-negative scalar curvature, first proven using minimal surfaces or inverse mean curvature flow, have been given new proofs using harmonic or $p$-harmonic functions. 
For example, new proofs of the positive mass theorem were given by Agostiniani-Mazzieri-Oronzio \cite{Agostiniani-Mazzieri-Oronzio}, Miao \cite{Miao-22}, and Bray-Kazaras-Khuri-Stern \cite{Bray-Kazaras-Khuri-Stern} using harmonic functions. Likewise, new proofs of the Penrose inequality (without rigidity) were given by Agostiniani-Mantegazza-Mazzieri-Oronzio \cite{Agostiniani-Mantegazza-Mazzieri-Oronzio} and Hirsch-Miao-Tam \cite{Hirsch-Miao-Tam-22} using $p$-harmonic functions. 

In this paper, we give a new proof of the Bray-Neves result on the $\sigma$-invariant of $\RP^3$ using harmonic functions. 
We adopt the same topological assumptions as in \cite{bray2004classification}.
\begin{definition}[Property A] 
A $3$-manifold $M^3$ has Property A if $M^3$ is not $S^3$ or a connect sum with an $S^2$ bundle over $S^1$ and $M^3$ can be expressed as $P^3\#Q^3$ where $P^3$ is prime and $Q^3$ is orientable.
\end{definition}

\begin{definition}[Property B]\label{property-B} 
    A $3$-manifold $M^3$ has Property B if $M^3$ is not $S^3$ or a connect sum with an $S^2$ bundle over $S^1$ and $M^3$ can be expressed as $P^3\#Q^3$ where $P^3$ is prime and $\alpha(Q_3)=2$, where $\alpha(Q)$ is the supremum of the Euler characteristic of the boundary (not necessarily connected) of all smooth connected regions (with two-sided boundaries) whose complements are also connected.
\end{definition}

\begin{remark}
    Property A implies Property B, see Lemma 2.10 in \cite{bray2004classification}. 
\end{remark}
Our main result is a new proof of the following theorem originally due to Bray and Neves. 

\begin{theorem}
\label{main-theorem} 
Assume that $M$ is a closed 3-manifold satisfying either Property A or Property B. Then $\sigma(M) \le \sigma_2$. 
\end{theorem}

As in \cite{bray2004classification}, this has the following corollary.

\begin{corollary}
The $\sigma$-invariant of $\RP^3$ satisfies $\sigma(\RP^3)=\sigma_2$. 
\end{corollary}

\subsection{Outline of proof}
We follow the same basic strategy as Bray and Neves, replacing inverse mean curvature flow by the level sets of a harmonic function. Given a manifold $(M^3,g)$ satisfying Property A or Property B, we first construct an associated asymptotically flat metric $g_{AF}$ on $M$ minus a point. This metric has vanishing scalar curvature, and it is equal to Schwarzschild in the model case of round $\RP^3$. Next, we transport the optimal test function from the Schwarzschild manifold to $(M-\{p\},g_{AF})$ along the level sets of a certain harmonic function.  
We then employ a monotonicity formula for harmonic functions
on asymptotically flat manifolds whose rigidity case is the spatial Schwarzschild manifold. 
Using that monotonicity formula in combination with an exponential quantity along the level sets of harmonic functions, we obtain a uniform upper bound of the Yamabe invariant for $M$, which completes the proof of Theorem \ref{main-theorem}.

\subsection{Organization} 
The rest of the paper is organized as follows. In Section \ref{notation}, we describe some notation for working with the level sets of a harmonic function on an asymptotically flat manifold. In Section \ref{monotonicity-section}, we discuss a monotonicity formula holding along the level sets of a harmonic function, which is in some sense analogous to the monotonicity of the Hawking mass. Finally, in Section \ref{main-proof-section}, we give the proof of the main theorem. 

\subsection{Acknowledgements} The authors would like to thank Professor Xin Zhou for his guidance and encouragement, and for many helpful discussions related to the problem. 

\section{Notation for Harmonic Functions}
\label{notation}

Let $(M^3,g)$ be an asymptotically flat manifold with non-empty, connected boundary $\Sigma$. We will always assume that $M$ has non-negative scalar curvature and that $H_2(M,\Sigma)=0$. Recall that for any $p \in [1,\infty)$, the $p$-Laplacian of a function $u$ is defined by 
\[
\lap_p u = \div(\vert \grad u\vert^{p-2}\grad u).
\]
In particular, the $2$-Laplacian is the usual Laplacian, and the 1-Laplacian 
\[
\lap_1 u = \div\left(\frac{\grad u}{\vert \grad u\vert}\right)
\]
computes the mean curvature of the level sets of $u$ (at points where this makes sense). 

Now suppose that for each $p\in (1,2]$, the function $u_p$ solves 
\[
\begin{cases}
\lap_p u_p = 0, &\text{on } M\\
u_p=1, &\text{on } \Sigma\\
u_p\to 0, &\text{at infinity}.
\end{cases}
\]
Moser \cite{Moser-07} observed that the function $w_p = -(p-1)\log(u_p)$ satisfies the equation 
\[
\lap_p w_p = \vert \grad w_p\vert^p. 
\]
Formally, taking the limit of this equation as $p\to 1$ gives $\lap_1 w_1 = \vert \grad w_1\vert$, which is the equation for inverse mean curvature flow. In fact, Moser showed that it is possible to take a weak limit of the functions $w_p$ as $p\to 1$, and this weak limit is a solution to Huisken and Ilmanen's weak inverse mean curvature flow. 

Motivated by this, we will employ the following notation. Let $u$ solve
\begin{equation}
\label{mian-equation}
\begin{cases}
    \lap u = 0, &\text{on } M\\
    u=1, &\text{on } \Sigma\\
    u\to 0, &\text{at infinity}.
\end{cases}
\end{equation} 
Let $w=-\log(u)$, and define the level sets $\Sigma_t = \{w=t\}$. Note in particular that $\Sigma_0 = \Sigma$. Define the quantities
\begin{gather}
\label{WB-equation} W(t) = \int_{\Sigma_t} \vert \grad w\vert^2 \,da
\quad\text{and}\quad
 \mathcal B(t) = e^t(4\pi - W(t)).
\end{gather}
Then $\mathcal B(t)$ coincides with the quantity defined by Miao in \cite{Miao-22}, up to a constant factor and reparameterization. In particular, as proven in \cite{Miao-22}, $\mathcal B(t)$ is a monotone increasing function of $t$.  Moreover, if $\mathcal B(t) = \mathcal B(0)$ for some $t > 0$, then $(M,g)$ is isometric to Euclidean space outside of a round ball (see \cite{Hirsch-Miao-Tam-22} Theorem 1.3(iii)). 

The reader should note that our notation differs from that commonly used in the literature, in which $u$ is the harmonic function on $M$ with $u=0$ on $\Sigma$ and $u\to 1$ at infinity, $t$ parameterizes the level sets of $u$, and the quantity $\mathcal B$ is defined directly in terms of $u$. We have chosen to use our notation because we feel it better facilitates the comparison between $w$ and the inverse mean curvature flow, $W$ and the Willmore energy, and $\mathcal B$ and the Hawking mass. We hope that this will not cause the reader any confusion.

Along the inverse mean curvature flow, area grows exponentially. We end this section by recording the corresponding exponential quantity along the level sets of harmonic functions. 

\begin{proposition}
\label{exponential-quantity}
One has 
\begin{align}
    \int_{\Sigma_t}|\nabla w|\, da=C_0e^t,
\end{align}
where $C_0 = \int_{\Sigma} \vert \grad w\vert\,da$.
\end{proposition}
\begin{proof}
    This follows from straightforward computation.
    First, compute the mean curvature of the level set, which is denoted as $H$. This gives
    \begin{align*}
        H&=\text{div}\left(\frac{\nabla w}{|\nabla w|}\right)\\
         &=\frac{\Delta w}{|\nabla w|}-\frac{\langle \nabla |\nabla w|,\nabla w\rangle}{|\nabla w|^2}\\
         &=|\nabla w|-\frac{\langle \nabla |\nabla w|,\nabla w\rangle}{|\nabla w|^2}.
    \end{align*}
    Then using the first variation, we obtain
    \begin{align*}
        \frac{d}{dt}\int_{\Sigma_t}|\nabla w|da&=\int_{\Sigma_t}Hda+\int_{\Sigma_t}\frac{\langle \nabla |\nabla w|,\nabla w\rangle}{|\nabla w|^2}da\\
        &=\int_{\Sigma_t}|\nabla w|da,
    \end{align*}
    and the proposition follows.
\end{proof}

\section{Monotonicity Formulas on Schwarzschild}
\label{monotonicity-section}
The goal of this section is to derive a monotonicity formula for harmonic functions on asymptotically flat manifolds, whose equality case is the spatial Schwarzschild manifold.   We note that Miao \cite{Miao-22} has already proven several theorems of this nature. Although he does not state the precise result we need, his arguments can be applied almost verbatim. We include the details for completeness.

The next Proposition is essentially Theorem 7.3 in \cite{Miao-22}, except that we compare $\mathcal B(0)$ to $\mathcal B(t)$ instead of comparing $\mathcal B(0)$ to $\lim_{t\to\infty} \mathcal B(t)$. 

\begin{proposition} 
\label{monotonicity}
Assume that $(M^3,g)$ is a complete asymptotically flat manifold with non-empty, connected boundary $\Sigma$. Assume that $M$ has non-negative scalar curvature and that $H_2(M,\Sigma)=0$. Let $u$ be the harmonic function on $M$ with $u=1$ on $\Sigma$ and $u\to 0$ at infinity. Define $w = -\log(u)$ and let $\Sigma_t = \{w=t\}$. Define $W$ as in (\ref{WB-equation}). 
Then for any $t\ge 0$ one has 
\begin{equation}
\label{monotonicity-equation}
 W(t) \le   \left[e^{-t}\sqrt{W(0)} + (1-e^{-t} )\sqrt{4\pi}\right]^2.
\end{equation}
Moreover, if equality holds for some $t > 0$, then $M$ is isometric to a spatial Schwarzschild manifold (possibly with negative mass) outside some rotationally symmetric sphere. 
\end{proposition} 

\begin{proof}
For any $k > 0$, define the function $v = u + k(1-u)$. Let $\bar g = k^{-4} v^4 g$ and let $\bar u = u/v$. It is easily verified that $\bar u$ is $\bar g$-harmonic and that $\bar u = 1$ on $\Sigma$ and $\bar u \to 0$ at infinity.  Moreover, $(M,\bar g)$ is still asymptotically flat with non-negative scalar curvature. Let $\bar w = -\log(\bar u)$. Then 
\begin{align*}
\grad \bar w = \frac{\grad v}{v}-\frac{\grad u}{u} = \left(\frac{1-k}{1+k(\frac{1-u}{u})}-1\right)\frac{\grad u}{u} = \left(\frac{k}{v}\right) \grad w. 
\end{align*}
It follows that 
\begin{align*}
\bar {\mathcal B}(s) &= e^s\left(4\pi -\int_{\bar \Sigma_s} \vert \bar \grad \bar w\vert^2 \, d\bar a\right)\\
&= e^s\left(4\pi -\int_{\bar \Sigma_s} \vert  \grad \bar w\vert^2 \, d a\right) = e^s\left(4\pi -\int_{\bar \Sigma_s} \left(\frac{k}{v}\right)^2\vert  \grad  w\vert^2 \, d a\right).
\end{align*}
This is an increasing function of $s \ge 0$. In particular, we have $\bar{\mathcal B}(s) \ge \bar{\mathcal B}(0)$, i.e.,  
\[
e^s\left(4\pi -\int_{\bar \Sigma_s}  \left(\frac{k}{v}\right)^2 \vert \grad  w\vert^2 \, d a\right) \ge 4\pi - k^2 \int_{\Sigma} \vert \grad w\vert^2\, da. 
\]
Note that $e^{-s} = \bar u = u/v$ and that $u$ and $v$ are both constant on $\bar \Sigma_s$. Therefore, for any $t\ge 0$,  the above inequality implies that 
\[
\int_{\Sigma_t} \vert \grad w\vert^2 \,da \le \frac{4\pi(1-u)v}{k^2} +uv  \int_{\Sigma}\vert \grad w\vert^2\, da, 
\]
where, on the right hand side, $u$ and $v$ are equal to their values on $\Sigma_t$.
Elementary calculus shows that the right hand side is minimized for 
$
k = \sqrt{{4\pi}/W(0)},
$
and this choice yields that 
\[
W(t) \le   \left[e^{-t}\sqrt{W(0)} + (1-e^{-t} )\sqrt{4\pi}\right]^2,
\]
which is the desired inequality. 
 
It remains to study the rigidity case. 
Assuming equality holds, we must have $k = \sqrt{4\pi/W (0)}$ and $\bar{\mathcal B}(s) = \bar{\mathcal B}(0)$ for this choice of $k$ and some $s > 0$. By rigidity for $\bar {\mathcal B}$, this implies that $(M,\bar g)$ is isometric to $(\R^3 - B_{r}(0), g_{\text{euc}})$ for some $r > 0$, and moreover that
\[
\frac{u}{u+k(1-u)} = \bar u = \frac{r}{\vert x\vert}. 
\]
This implies that 
\[
u = \frac{k\bar u}{1+(k-1)\bar u} \quad \text{and} \quad v = \frac{k}{1+(k-1)\bar u} =  \frac{k\vert x\vert}{r(k-1)+\vert x\vert}. 
\]
It follows that
\[
g = k^4 v^{-4} \bar g =  \left(1+\frac{r(k-1)}{\vert x\vert}\right)^4 g_{\text{euc}},
\]
and so $(M,g)$ is isometric to a Schwarzschild manifold of mass $2r(k-1)$ outside of a rotationally symmetric sphere. 

Finally, it is easy to check by explicit computation that equality holds on a Schwarzschild manifold.  For a given $m\in \R$ and $r > 0$, consider the metric 
\[
\T g = \left(1+\frac{m}{2\vert x\vert}\right)^4 g_{\text{euc}}.
\]
on $\R^3 - B_r(0)$. If $m$ is negative, then we require $r > \vert m\vert/2$. 
Let $\Sigma = \bd B_{r}$.  The $\T g$-harmonic function with $u=1$ on $\Sigma$ and $u\to 0$ at infinity is 
\[
u(x) = \frac{m+2r}{m+2\vert x\vert}. 
\]
It follows that $
\vert \grad w\vert = \frac{2}{m+2\vert x\vert}$
and hence that 
\[
W(0) = \int_{\Sigma} \vert \T \grad w\vert^2 \, d\T a = \int_{\Sigma} \vert \grad w\vert^2 \, da = \frac{16\pi r^2}{(m+2r)^2}. 
\]
More generally, note that 
\[
\Sigma_t = \left\{\vert x\vert = \frac{e^t(m+2r)-m}{2}\right\}
\]
and so 
\[
W(t) = \int_{\Sigma_t} \vert \grad w\vert^2 \, da=  \frac{4\pi}{(m+2r)^2}\left(m+2r-me^{-t}\right)^2.
\]
One can also compute that
\[
\left[e^{-t} \sqrt{W(0)} + (1-e^{-t})\sqrt{4\pi}\right]^2 = \frac{4\pi}{(m+2r)^2}\left(m+2r-me^{-t}\right)^2,
\]
and so equality holds. 
\end{proof}

In order to effectively apply this monotonicity formula, we need to be able to estimate $W(0)$. Fortunately, this can be done using another theorem of Miao. The following is Corollary 7.1 in \cite{Miao-22}, adapted to our notation.

\begin{proposition}
    \label{initial-estimate}
    With the same notation and assumptions as Proposition \ref{monotonicity}, one has 
    \[
        \left(\frac{1}{\pi} \int_{\Sigma} \vert \grad w\vert^2\, da\right)^{1/2} \le \left(\frac{1}{16\pi}\int_{\Sigma}H^2\, da\right)^{1/2}+1.
    \]
    Moreover, equality holds if and only if $(M,g)$ is isometric to a Schwarzschild manifold outside a rotationally symmetric sphere with non-negative constant mean curvature. 
\end{proposition}

This has the following corollary.

\begin{corollary}\label{minimal-initial-data-estimate}
    With the same notation and assumptions as Proposition \ref{monotonicity}, assume in addition that $\Sigma$ is minimal. Then 
    \[
        W(t) \le \pi\left(2-e^{-t}\right)^2
    \]
    for all $t\ge 0$. Equality holds for some $t \ge 0$ if and only if $(M,g)$ is isometric to a spatial Schwarzschild manifold with positive mass and $\Sigma$ is the horizon. 
\end{corollary}

\begin{proof}
    Since $\Sigma$ is minimal, Proposition \ref{initial-estimate} implies that $W(0) \le \pi$. Substituting this into equation (\ref{monotonicity-equation}) and simplifying shows that 
    \[
    W(t) \le \pi(2-e^{-t})^2. 
    \]
    If equality holds for some $t \ge 0$ then $W(0)=\pi$, and so Proposition \ref{initial-estimate} implies that $(M,g)$ is isometric to a spatial Schwarzschild manifold with $\Sigma$ as the horizon. 
\end{proof}
\section{Application of the monotonicity formula}
\label{main-proof-section}

In this section, we use Corollary \ref{minimal-initial-data-estimate} to give a new proof of Bray-Neves' result \cite{bray2004classification}.
We begin by introducing some definitions and basic facts.

Suppose $(M^3,g)$ is a compact $3$-manifold with $Y(M,[g])>0$, and $M$ satisfies Property B \ref{property-B}. Then it cannot be conformally equivalent to 
the $3$-sphere with round metric, hence the Yamabe constant can be achieved by some positive constant scalar curvature metric $g_0\in [g]$.
Working inside $(M,g_0)$ now, we define the conformal Laplacian operator as
\begin{align}
    L_0\equiv \Delta_{g_0}-\frac{1}{8}R_{g_0},
\end{align}
where $R_{g_0}$ is the constant positive scalar curvature of $(M,g_0)$.

Now choose any $p\in M$, and define $G_p(x)$ to be the Green's function of $(M,g_0)$ if it satisfies
\begin{align}
    L_0G_p(x)=0,\quad \forall x\in M\setminus\{p\},
\end{align}
and
\begin{align}
    \lim_{q\to p}d(p,q)G_{p}(q)=1.
\end{align}  
The Green's function exists and it is positive from the positivity of $R_0$ and the Maximal Principal.

Let $g_{AF}=G_p(x)^4g_0$ on $M\setminus \{p\}$. If we denote $M_p=M\setminus \{p\}$, then $(M\setminus\{p\},g_{AF})$ is the asymptotically 
flat $3$-manifold with zero scalar curvature, and $p$ is its infinity.
Note that
\begin{align}
    Y(M,[g])=Y(M,[g_0])=Y(M_p,[g_{AF}]),
\end{align}
and by definition
\begin{align}
    Y(M_p,[g_{AF}])=\inf\left\{\frac{\int_{M_p}8|\nabla \phi|^2dV}{(\int_{M_p}\phi^6dV)^{1/3}}: \phi\in H_{c}^1(M_p,g_{AF})\right\}.
\end{align}

\begin{remark}
    It is actually sufficient to assume that $\phi\in H^1_{\text{loc}}(M_p,g_{AF}) \cap L^6(M_p,g_{AF})$ with $\lim_{x\to \infty} \phi(x)\vert x\vert^{1/2} = 0$ when taking the above infimum. 
\end{remark}

We can now give the proof of the main theorem. 

\begin{restatetheorem}
Assume that $M$ is a closed 3-manifold satisfying either Property A or Property B. Then $\sigma(M) \le \sigma_2$. 
\end{restatetheorem}

\begin{proof}
    We first consider the model case, when $M=\mathbb{RP}^3$, and $g_0$ is the standard round metric. As in \cite{bray2004classification}, $(\RP^3_p,g_{AF})$ can be identified with a spatial Schwarzschild manifold with antipodal identification on the horizon. Let $w_s=-\log(u_s)$, where we use $u_s$ to 
    denote the solution to \eqref{mian-equation} on $(\mathbb{RP}^3_{p},g_{AF})$. Then define
\begin{align}
    f(t)=u_0(\Sigma_t^s),
\end{align}
where $\Sigma_t^s=\{x\in (\mathbb{RP}^3_p,g_{AF}): w_s(x)=t\}$, and $u_0(x)=G_p(x)^{-1}$. Let $\phi_s=f\circ w_s$. As in \cite{bray2004classification},  Obata's theorem implies that $\phi_s$ is the optimal test function in the model and 
\[
\frac{\int_{\RP^3_p} 8\vert \grad \phi_s\vert^2 \, dV}{\left(\int_{\RP^3_p} \phi_s^6\, dV\right)^{1/3}} = \sigma_2. 
\]
Next we aim to transport this test function to an arbitrary $M$. 

For any $(M,g_0)$ satisfying Property A or Property B, adopting the same argument as in \cite{bray2004classification},
we know there exists an outermost minimizing $\Sigma\subset M_p$, which is the boundary of a minimizing hull. Using that as 
an initial condition, let $w=-\log(u)$, where $u$ is the solution to \eqref{mian-equation}. 
We can construct a test function $\phi\in H_{\text{loc}}^1(M_p,g_{AF})$ via
\begin{align}
    \phi(x)=f(w(x)).
\end{align}
It remains to estimate 
\[
\frac{\int_{M_p}8|\nabla \phi|^2dV}{(\int_{M_p}\phi^6dV)^{1/3}}.
\]
Estimating the numerator with the coarea formula, one obtains
\begin{equation}\label{numerator-estimate}
\begin{split}
    \int_{M_p}|\nabla\phi|^2dV&=\int_{M_p}f'(w(x))^2|\nabla w|^2dV\\
                              &=\int_0^{\infty}f'(t)^2(\int_{\Sigma_t}|\nabla w|da)dt\\
  &=C_0\int_0^{\infty}f'(t)e^tdt,
  \end{split}
\end{equation}
where $C_0=\int_{\Sigma_0}|\nabla w|\, da$, and we used Proposition \ref{exponential-quantity}.

 For the denominator, we need to use the coarea formula, H\"older's inequality, and   Corollary \ref{minimal-initial-data-estimate}. This gives 
\begin{equation}\label{denominator-estimate}
\begin{split}
    \int_{M_p}\phi^6\, dV&=\int_{M_p}f(w(x))^6\, dV\\
&=\int_0^{\infty}f(t)^6(\int_{\Sigma_t}|\nabla w|^{-1}\, da)\, dt\\
&\geq \int_0^{\infty}f(t)^6(\int_{\Sigma_t}|\nabla w|^2da)^{-2}(\int_{\Sigma_t}|\nabla w|da)^3\, dt\\
&\geq C_0^3\int_0^{\infty}f(t)^6e^{3t}W(t)^{-2}\, dt\\
&\geq \pi^{{-2}}C_0^3\int_0^{\infty}f(t)^6e^{3t}(2-e^{-t})^{-4}\, dt.
\end{split}
\end{equation}
Combining the previous estimates, we obtain
\begin{equation}
    \frac{\int_{M_p}8|\nabla \phi|^2dV}{(\int_{M_p}\phi^6dV)^{1/3}}\leq 
    \frac{8\int_0^{\infty}f'(t)e^tdt}{(\pi^{-2}\int_0^{\infty}f(t)^6e^{3t}(2-e^{-t})^{-4})^{1/3}}=\sigma_2.
\end{equation}
The second equality holds because all of the inequalities in \eqref{numerator-estimate} and \eqref{denominator-estimate} become equality in the model case. 
\end{proof}
\begin{remark}
For \eqref{denominator-estimate}, one needs to use the Lebesgue Monotone Convergence Theorem as in \cite{bray2004classification}.
\end{remark}
\bibliographystyle{plain}
\bibliography{biblio.bib}

\end{document}